\documentclass[11pt]{article}

\usepackage{epsfig}

\usepackage{amsmath,makeidx,amssymb,amscd,amsfonts}
\usepackage{epsfig,color,psfrag}
\usepackage[dvips, bookmarksopen, colorlinks,
linkcolor = blue, urlcolor = red, citecolor = red, menucolor = blue]{hyperref}

\newcommand\norm[1]{\|#1\|}
\newcommand\set[1]{\{#1\}}

\newtheorem{theorem}{Theorem}[section]
\newtheorem{lemma}[theorem]{Lemma}

\newtheorem{remark}[theorem]{Remark}

\begin{document}
\setcounter{footnote}{1}

\title{Iterative Regularization Methods for a Discrete Inverse Problem in MRI}

\author{
A.\,Leit\~ao%
\thanks{Department of Mathematics, Federal University of St.\,Catarina,
P.O. Box 476, 88040-900 Florian\'opolis, Brazil (aleitao@mtm.ufsc.br)}
\ and \ 
J.P.\,Zubelli%
\thanks{IMPA, Estr. Dona Castorina 110, 22460-320 Rio de Janeiro, Brazil
(zubelli@impa.br)}
}

\date{\small\today}

\maketitle

\begin{small}
\noindent {\bf Abstract:}
We propose and investigate efficient numerical methods for inverse problems
related to Magnetic Resonance Imaging (MRI).
Our goal is to extend the recent convergence results for the Landweber-%
Kaczmarz method obtained in \cite{HLS07}, in order to derive a convergent
iterative regularization method for an inverse problem in MRI.
\bigskip

\end{small}

\section{Introduction} \label{sec:intro}

{\em Magnetic Resonance Imaging}, also known as MR--Imaging or simply
MRI, is a noninvasive technique used in medical imaging to visualize body structures and
functions, providing detailed images in arbitrary planes.
Unlike {\em X-Ray Tomography} it does not use ionizing radiation, but uses
a powerful magnetic field to align the magnetization of hydrogen atoms in
the body. Radio waves are used to systematically alter the alignment of this
magnetization, causing the hydrogen atoms to produce a rotating magnetic
field detectable by the scanner.

More specifically, when a subject is in the scanner, the hydrogen nuclei (i.e., protons, found in abundance in the human body as water) align with the strong magnetic field. A
radio wave at the correct frequency for the protons to absorb energy pushes
some of the protons out of alignment. The protons then snap back to alignment,
producing a detectable rotating magnetic field as they do so. Since protons in
different areas of the body (e.g., fat and muscle) realign at different speeds,
the different structures of the body can be revealed.


The image to be identified in MRI corresponds to a complex valued function
${\cal P}: [0,1] \times [0,1] \to \mathbb C$ and the image acquisition
process is performed by so-called {\em receivers}.
Due to the physical nature of the acquisition process, the information
gained by the receivers does not correspond to the unknown image, but
instead, to ${\cal P}$ multiplied by receiver dependent {\em sensitivity
kernels}. In real life applications, the sensitivity kernels are not
precisely known and have to be identified together with ${\cal P}$. This corresponds
to a version of the ubiquitous deconvolution problem that has been investigated
by many authors. See for example~\cite{bertero,natwue,zmsh2003}

Our main goal in this article is to investigate efficient iterative methods
of Kaczmarz type for the identification problem related to MRI. We extend
the convergence results for the {\em loping Landweber-Kaczmarz} method in
\cite{HLS07} and derive a convergent iterative regularization method for
this inverse problem.

This article is outlined as follows.
In Section~\ref{sec:model-d} the description of a discrete mathematical
model for Magnetic Resonance Imaging is presented.
In Section~\ref{sec:model-i} we derive the corresponding inverse problem
for MRI.
In Section~\ref{sec:regul-iter} we investigate efficient iterative
regularization methods for this inverse problem. Using a particular
hypothesis on the sensitivity kernels, we are able to derive convergence
and stability results for the proposed iterative methods.

\section{The direct problem} \label{sec:model-d}

In what follows we present a discrete model for MRI. In our approach,
we follow the notation introduced in \cite{BK07}.
The image to be identified is considered to be a discrete function
$$
{\cal P}: \{1, \dots, p_{\mathrm{hor}}\} \times \{1, \dots, p_{\mathrm{ver}}\} =: {\mathbb B}
          \to {\mathbb C} \, ,
$$
where $p_{\mathrm{hor}}$, $p_{\mathrm{ver}} \in \mathbb N_0$ are known. Therefore, the number
of degrees of freedom related to this parameter is $p_{\mathrm{ num}} := p_{\mathrm{hor}} \times
p_{\mathrm{ver}}$ (typical values are $p_{\mathrm{hor}} = p_{\mathrm{ver}} = 256$; $p_{\mathrm{ num}} = 65536$).

As mentioned above, the image acquisition process is performed by several {\em receivers}, denoted
here by ${\cal R}_j$, $j = 0, \dots, r_{\mathrm{ num}}-1$, where $r_{\mathrm{ num}} \in \mathbb N_0$
is given (typically one faces the situation where $r_{\mathrm{ num}} << p_{\mathrm{ num}}$).
Due to the physical nature of the acquisition process, the information gained
by the receivers does not correspond to the unknown image, but instead, to
${\cal P}$ multiplied by receiver dependent {\em sensitivity kernels}
$$
{\cal S}_j \ = \ {\cal S}({\cal R}_j): {\mathbb B} \to {\mathbb C} \, , \
j = 0, \dots, r_{\mathrm{ num}}-1 \, .
$$

In real life applications, the sensitivity kernels ${\cal S}_j$ are not
precisely known and have to be identified together with ${\cal P}$. This
fact justifies the following ansatz:

\begin{itemize}
\item[(A1)] The sensitivity kernels ${\cal S}_j$ can be written as linear
combination of the given basis functions ${\mathcal B}_n : {\mathbb B} \to
{\mathbb C}$, for $n = 1, \dots, b_{\mathrm{ num}}$, and $b_{\mathrm{ num}} \in \mathbb N_0$.
\end{itemize}

\noindent
In other words, we assume the existence of coefficients $b_{j,n} \in \mathbb C$
such that
\begin{equation} \label{eq:sk-repr}
{\cal S}_j(m) \ = \ \sum_{n=1}^{b_{\mathrm{ num}}} b_{j,n} \, {\mathcal B}_n(m) \, , \
m \in {\mathbb B} \, , \ j = 0, \dots, r_{\mathrm{ num}}-1 \, .
\end{equation}
In the sequel we make use the abbreviated notations ${\mathbf b_j} :=
(b_{j,n})_{n=1}^{b_{\mathrm{ num}}}$ and  $({\mathbf b_j}) :=
({\mathbf b_j})_{j=0}^{r_{\mathrm{ num}}-1}$. Notice that the coefficient vectors
${\mathbf b_j} = {\mathbf b}({\cal R}_j)$ are receiver dependent.

The measured data for the inverse problem is given in a subset of the
Fourier space of the image ${\cal P}$, i.e. there exists a known subset
${\mathbb M} \subset {\mathbb B}$ (consisting of $p_{proj}$ elements) such
that the receiver dependent measurement ${\cal M}_j = {\cal M}(R_j)$ satisfies
$$
{\cal M}_j \ := \ {\mathbf P}[ {\cal F} ({\cal P} \times {\cal S}_j) ] \, , \
j = 0, \dots, r_{\mathrm{ num}}-1 \, .
$$
where ${\cal F}$ is the {\em Discrete Fourier Transform} (DFT) operator
defined by
\begin{eqnarray*}
{\cal F} : \{ f\, | \ f: {\mathbb B} \to {\mathbb C} \} & \to &
           \{ \hat{f}\, | \ \hat{f}: {\mathbb B} \to {\mathbb C} \} \\
f & \mapsto & ({\cal F}(f))(m) := \sum_{n=0}^{p_{\mathrm{ num}}-1} f(n) \,
              \exp\Big( - \frac{2\pi i}{p_{\mathrm{ num}}} nm \Big) \, ,
\end{eqnarray*}
and $\mathbf P$ is the operator defined by
\begin{eqnarray*}
{\mathbf P} : \big\{ f\, | \ f: {\mathbb B} \to {\mathbb C} \big\} & \to &
              \big\{ g\, | \ g: {\mathbb M} \to {\mathbb C} \big\} =: Y \\
f & \mapsto & ({\mathbf P}[f])(m) := f(m)\, ,\ m \in {\mathbb M} \, .
\end{eqnarray*}

Notice that, due to ansatz (A1) and the linearity of ${\cal F}$
and ${\mathbf P}$, the measured data ${\cal M}_j \in Y$ can be written
in the form
\begin{equation} \label{eq:def-measurmnt}
{\cal M}_j \ = \ \sum_{n=1}^{b_{\mathrm{ num}}} b_{j,n} \,
                {\mathbf P}[ {\cal F} ({\cal P} \times {\mathcal B}_n) ]
                \, , \ j = 0, \dots, r_{\mathrm{ num}}-1 \, .
\end{equation}

\begin{remark}
The numerical evaluation of the DFT requires naively $O(p_{\mathrm{ num}}^2)$ arithmetical
operations. 
However, in practice 
the DFT must be replaced by
the {\em Fast Fourier Transform} (FFT), which can be computed by the
Cooley-Tukey algorithm%
\footnote{The FFT algorithm was published independently by J.W.~Cooley and
J.W.~Tukey in 1965. However, this algorithm was already known to C.F.~Gauss
around 1805.}
and requires only $O(p_{\mathrm{ num}} \log(p_{\mathrm{ num}}))$ operations.
\end{remark}

\section{The inverse problem} \label{sec:model-i}

Next we use the discrete model discussed in the previous section as a 
starting point to formulate an inverse problem for MRI.

The unknown parameters to be identified are the discrete image function
${\cal P}$ and the sensitivity kernels ${\cal S}_j$. Due to the ansatz
(A1), the parameter space $X$ consists of pairs of the form
$({\cal P}, ({\mathbf b_j}))$, i.e.
$$
X \ := \ \big\{ ({\cal P}, ({\mathbf b_j}) ) \, ;
       \ {\cal P} \in {\mathbb C}^{p_{\mathrm{ num}}} \, ,
       \ ({\mathbf b_j}) \in ({\mathbb C}^{b_{\mathrm{ num}}})^{r_{\mathrm{ num}}} \big\} \, .
$$
It is immediate to observe that $X$ can be identified with
${\mathbb C}^{(p_{\mathrm{ num}} + {b_{\mathrm{ num}}}\times{r_{\mathrm{ num}}})}$, while $Y$ can
be identified with ${\mathbb C}^{p_{proj}}$.

The {\em parameter to output} operators $F_i: X \to Y$ are defined by
\begin{equation} \label{eq:def-fi}
F_i: ( {\cal P}, ({\mathbf b_j}) ) \mapsto
   \ \sum_{n=1}^{b_{\mathrm{ num}}} b_{i,n} \, {\mathbf P}
   [ {\cal F} ({\cal P} \times {\mathcal B}_n) ]
   \, , \ i = 0, \dots, r_{\mathrm{ num}}-1 \, .
\end{equation}

Due to the experimental nature of the data acquisition process, we shall
assume that the data ${\cal M}_i$ in \eqref{eq:def-measurmnt} is not exactly
known. Instead, we have only approximate measured data ${\cal M}_i^\delta
\in Y$ satisfying
\begin{equation}\label{eq:noisy-data}
\| {\cal M}_i^\delta - {\cal M}_i \| \le \delta_i \, ,
\end{equation}
with $\delta_i > 0$ (noise level).
Therefore, the inverse problem for MRI can be written in the form of the
following system of nonlinear equations
\begin{equation} \label{eq:inv-probl}
F_i( {\cal P}, ({\mathbf b_j}) ) \ = \ {\cal M}_i^\delta \, ,
                                     \ i = 0, \dots, r_{\mathrm{ num}}-1 \, .
\end{equation}

It is worth noticing that the nonlinear operators $F_i$'s are continuously
Fr\'echet differentiable, and the derivatives are locally Lipschitz
continuous.

\section{Iterative regularization} \label{sec:regul-iter}

In this section we analyze efficient iterative methods for obtaining stable
solutions of the inverse problem in \eqref{eq:inv-probl}.

\subsection{An image identification problem} \label{ssec:regul-iter-im}

Our first goal is to consider a simplified version of problem
\eqref{eq:inv-probl}. We assume that the sensitivity kernels ${\cal S}_j$
are known, and we have to deal with the problem of determining only the
image ${\cal P}$. This assumption can be justified by the fact that, in practice, one has very good approximations for the sensitivity kernels, while
the image ${\cal P}$ is completely unknown.

In this particular case, the inverse problem reduces to
\begin{equation} \label{eq:inv-probl-s}
\tilde{F}_i( {\cal P} ) \ = \ {\cal M}_i^\delta \, ,\ i = 0, \dots, r_{\mathrm{ num}}-1\, ,
\end{equation}
where $\tilde{F}_i( {\cal P} ) = F_i( {\cal P}, ({\mathbf b_j}) )$, the
coefficients $({\mathbf b_j})$ being known. This is a much simpler problem,
since $\tilde{F}_i: \tilde{X} \to Y$ are linear and bounded operators,
defined at $\tilde{X} := \{f\, |\ f: {\mathbb B}\to{\mathbb C}\}$.

We follow the approaches in \cite{HLS07,CHL08} and derive two iterative
regularization methods of Kaczmarz type for problem \eqref{eq:inv-probl-s}.
Both iterations can be written in the form
\begin{equation} \label{eq:def-iter}
    {\cal P}_{k+1}^\delta = {\cal P}_{k}^\delta - \omega_k \alpha_k s_k \, ,
\end{equation}
where
\begin{equation} \label{eq:def-sk}
s_k := \tilde{F}_{[k]}({\cal P}_{k}^\delta)^*
       ( \tilde{F}_{[k]}({\cal P}_{k}^\delta) - {\cal M}_i^\delta ) \, ,
\end{equation}
\begin{equation} \label{eq:def-omk}
\omega_k := \begin{cases}
           1  & \norm{\tilde{F}_{[k]}({\cal P}_{k}^\delta) - {\cal M}_i^\delta}
                > \tau \delta_{[k]} \\
           0  & \text{otherwise}
         \end{cases} \, .
\end{equation}
Here $\tau > 2$ is an appropriately chosen constant,
$[k] := (k \mod r_{\mathrm{ num}}) \in \set{0, \dots, r_{\mathrm{ num}}-1}$ (a group of
$r_{\mathrm{ num}}$ subsequent steps, starting at some multiple $k$ of $r_{\mathrm{ num}}$,
is called a {\em cycle}),
${\cal P}_0^\delta = {\cal P}_0 \in \tilde{X}$ is an initial guess,
possibly incorporating some \textit{a priori} knowledge about the
exact image, and $\alpha_k \ge 0$ are relaxation parameters.

Distinct choices for the relaxation parameters $\alpha_k$ lead to
the definition of the two iterative methods:
\begin{itemize}
\item[1)] If $\alpha_k$ is defined by
\begin{equation} \label{eq:def-alk}
\alpha_k :=
    \begin{cases}
      \| s_k \|^2 / \| \tilde{F}_{[k]}({\cal P}_{k}^\delta)
         s_k \|^2 & \omega_k = 1 \\
      1           & \omega_k = 0
   \end{cases} \, ,
\end{equation}
the iteration \eqref{eq:def-iter} corresponds to the
\textit{loping Steepest-Descent Kaczmarz method} (lSDK) \cite{CHL08}.
\item[2)] If $\alpha_k \equiv 1$, the iteration \eqref{eq:def-iter}
corresponds to the \textit{loping Landweber-Kaczmarz method} (lLK)
\cite{HLS07}.
\end{itemize}

The iterations should be terminated when, for the first time, all
${\cal P}_k$ are equal within a cycle. That is, we stop the iteration at
the index $k_*^\delta$, which is the smallest multiple of $r_{\mathrm{ num}}$ such that
\begin{equation} \label{eq:def-discrep}
{\cal P}_{k_*^\delta} = {\cal P}_{k_*^\delta+1} = \dots =
{\cal P}_{k_*^\delta+r_{\mathrm{ num}}-1} \, .
\end{equation}

\subsubsection*{Convergence analysis of the lSDK method}

From \eqref{eq:def-fi} follows that the operators $\tilde{F}_i$ are linear
and bounded. Therefore, there exist $M > 0$ such that
\begin{equation} \label{eq:a-dfb}
\| \tilde{F}_i \| \le M \, , \ i = 0, \dots, r_{\mathrm{ num}}-1 \, .
\end{equation}
Since the operators $\tilde{F}_i$ are linear, the {\em local tangential cone 
condition} is trivially satisfied (see \eqref{eq:a-tcc} below). Thus, the
constant $\tau$ in \eqref{eq:def-omk} can be chosen such that $\tau > 2$.
Moreover, we assume the existence of
\begin{equation} \label{eq:a-xstar}
{\cal P}^* \in B_{\rho/2}({\cal P}_0) \ \ {\rm such \ that}
\ \ \tilde{F}_i({\cal P}^*) = {\cal M}_i \, , \ i = 0, \dots, r_{\mathrm{ num}}-1 \, ,
\end{equation}
where $\rho > 0$ and $( {\cal M}_i )_{i=0}^{r_{\mathrm{ num}}-1} \in Y^{r_{\mathrm{ num}}}$
denotes to exact data satisfying \eqref{eq:noisy-data}.

In the sequel we summarize several properties of the lSDK iteration.
For a complete prof of the results we refer the reader to
\cite[Section~2]{CHL08}.

\begin{lemma}
Let the coefficients $\alpha_k$ be defined as in \eqref{eq:def-alk},
the assumption \eqref{eq:a-xstar} be satisfied for some ${\cal P}^* \in
\tilde{X}$, and the stopping index $k_*^\delta$ be defined as in
\eqref{eq:def-discrep}. Then, the following assertions hold:
\begin{itemize}
\item[1)] There exists a constant $\underline{\alpha} > 0$ such that
$\alpha_k > \underline{\alpha}$, for $k = 0, \dots, k_*^\delta$.
\item[2)] Let $\delta_i > 0$ be defined as in \eqref{eq:noisy-data}. Then the stopping
index $k_*^\delta$ defined in \eqref{eq:def-discrep} is finite.
\item[3)] ${\cal P}_k^\delta \in B_{\rho/2}({\cal P}_0)$ for all
$k \le k_*^\delta$.
\item[4)] The following monotony property is satisfied:
\begin{eqnarray}
\| {\cal P}_{k+1}^\delta - {\cal P}^* \|^2 & \le &
     \| {\cal P}_k^\delta - {\cal P}^* \|^2 \, , \
     k = 0, 1, \dots, k_*^\delta \, , \label{eq:sdk-monot1} \\
\| {\cal P}_{k+1}^\delta - {\cal P}^* \|^2 & = &
     \| {\cal P}_k^\delta - {\cal P}^* \|^2 \, , \
     k > k_*^\delta \, . \label{eq:sdk-monot2}
\end{eqnarray}
Moreover, in the case of noisy data (i.e. $\delta_i > 0$) we have
\begin{equation} \label{eq:sdk-monot-res}
\norm{ \tilde{F}_{i}({\cal P}_{k_*^\delta}^\delta) - {\cal M}_i^\delta} \leq
\tau \delta_i \, , \ i = 0, \dots, r_{\mathrm{ num}}-1 \, .
\end{equation}
\end{itemize}
\end{lemma}

Next we prove that the lSDK method is a convergent regularization method
in the sense of \cite{EHN96}.

\begin{theorem}[Convergence] \label{th:converg}
Let $\alpha_k$ be defined as in \eqref{eq:def-alk}, the assumption
\eqref{eq:a-xstar} be satisfied for some ${\cal P}^* \in \tilde{X}$,
and the data be exact, i.e. ${\cal M}_i^\delta = {\cal M}_i$ in
\eqref{eq:noisy-data}.
Then, the sequence ${\cal P}_k^\delta$ defined in \eqref{eq:def-iter}
converges to a solution of \eqref{eq:inv-probl-s} as $k \to \infty$.
\end{theorem}
\begin{proof}
Notice that, since the data is exact, we have $\omega_k = 1$ for all
$k > 0$. The proof follows from \cite[Theorem~3.5]{CHL08}.
\end{proof}

\begin{theorem}[Stability] \label{th:stabil}
Let the coefficients $\alpha_k$ be defined as in \eqref{eq:def-alk}, and
the assumption \eqref{eq:a-xstar} be satisfied for some ${\cal P}^* \in
\tilde{X}$. Moreover, let the sequence
$\{ (\delta_{1,m}, \dots, \delta_{r_{\mathrm{ num}},m}) \}_{m\in\mathbb{N}}$
(or simply $\{ \mathbf{\delta_m} \}_{m\in\mathbb{N}}$) be such that
$\lim_{m\to\infty} ( \max_i \delta_{i,m} ) = 0$, and let
${\cal M}_i^{\mathbf{\delta_m}}$ be a corresponding sequence of noisy
data satisfying \eqref{eq:noisy-data} (i.e.
$\norm{ {\cal M}_i^{\mathbf{\delta_m}} - {\cal M}_i } \le \delta_{i,m}$,
$i = 0, \dots, r_{\mathrm{ num}}-1$, $m \in \mathbb N)$.
For each $m \in \mathbb N$, let $k_*^m$ be the stopping index defined in
\eqref{eq:def-discrep}.
Then, the lSDK iterates ${\cal P}_{k_*^m}^\mathbf{\delta_m}$ converge to
a solution of \eqref{eq:inv-probl-s} as $m \to \infty$.
\end{theorem}
\begin{proof}
The proof follows from \cite[Theorem~3.6]{CHL08}.
\end{proof}

\subsubsection*{Convergence analysis of the lLK method}

The convergence analysis results for the lLK iteration are analog to the
ones presented in Theorems~\ref{th:converg} and~\ref{th:stabil} for the
lSDK method. In the sequel we summarize the main results that we could
extend to the lLK iteration.

\begin{theorem}[Convergence Analysis] \label{th:converg-anal}
Let $\alpha_k \equiv 1$, the assumption \eqref{eq:a-xstar} be satisfied
for some ${\cal P}^* \in \tilde{X}$, the operators $\tilde{F}_i$ satisfy
\eqref{eq:a-dfb} with $M = 1$, and the stopping index $k_*^\delta$ be defined
as in \eqref{eq:def-discrep}. Then, the following assertions hold:
\begin{itemize}
\item[1)] Let $\delta_i > 0$ as in \eqref{eq:noisy-data}. Then the stopping
index $k_*^\delta$ defined in \eqref{eq:def-discrep} is finite.
\item[2)] ${\cal P}_k^\delta \in B_{\rho/2}({\cal P}_0)$ for all
$k \le k_*^\delta$.
\item[3)] The monotony property in \eqref{eq:sdk-monot1}, \eqref{eq:sdk-monot2}
is satisfied. Moreover, in the case of noisy data, \eqref{eq:sdk-monot-res}
holds true.
\item[4)] For exact data, i.e. $\delta_i = 0$ in \eqref{eq:noisy-data}, the
sequence ${\cal P}_k^\delta$ defined in \eqref{eq:def-iter} converges to a
solution of \eqref{eq:inv-probl-s} as $k \to \infty$.
\item[5)] Let the sequence $\{ \mathbf{\delta_m} \}_{m\in\mathbb{N}}$, the
corresponding sequence of noisy data ${\cal M}_i^{\mathbf{\delta_m}}$, and
the stopping indexes $k_*^m$ be defined as in Theorem~\ref{th:stabil}.
Then, the lLK iterates ${\cal P}_{k_*^m}^\mathbf{\delta_m}$ converge to
a solution of \eqref{eq:inv-probl-s} as $m \to \infty$.
\end{itemize}
\end{theorem}
\begin{proof}
The proof follows from corresponding results for the lLK iteration for
systems of nonlinear equations in \cite{HLS07}.
\end{proof}

Notice that the assumption $M = 1$ in Theorem~\ref{th:converg-anal}
is nonrestrictive. Indeed, since the operators $\tilde{F}_i$ are linear,
it is enough to scale the equations in \eqref{eq:inv-probl-s} with
appropriate multiplicative constants.

\subsection{Identification of  image and sensitivity} \label{ssec:regul-iter-ims}

Our next goal is to consider the problem of determining both the image
${\cal P}$ as well as the sensitivity kernels ${\cal S}_j$ in
\eqref{eq:inv-probl}. The lLK and lSDK iterations can be extended to the
nonlinear system in a straightforward way

\begin{equation} \label{eq:def-iter-nl}
( {\cal P}_{k+1}^\delta , ({\mathbf b_j})_{k+1}^\delta ) =
( {\cal P}_{k}^\delta , ({\mathbf b_j})_{k}^\delta ) - \omega_k \alpha_k s_k \, ,
\end{equation}
where
\begin{equation} \label{eq:def-sk-nl}
s_k := F'_{[k]}( {\cal P}_{k}^\delta , ({\mathbf b_j})_{k}^\delta )^*
       ( F_{[k]}( {\cal P}_{k}^\delta , ({\mathbf b_j})_{k}^\delta )
         - {\cal M}_i^\delta ) \, ,
\end{equation}
\begin{equation} \label{eq:def-omk-nl}
\omega_k := \begin{cases}
           1  & \norm{F_{[k]}( {\cal P}_{k}^\delta , ({\mathbf b_j})_{k}^\delta )
                - {\cal M}_i^\delta} > \tau \delta_{[k]} \\
           0  & \text{otherwise}
         \end{cases} \, .
\end{equation}
In the lLK iteration we choose $\alpha_k \equiv 1$, and in the lSDK iteration
we choose
\begin{equation} \label{eq:def-alk-nl}
\alpha_k :=
    \begin{cases}
      \| s_k \|^2 / \| F'_{[k]}( {\cal P}_{k}^\delta , ({\mathbf b_j})_{k}^\delta )
         s_k \|^2 & \omega_k = 1 \\
      1           & \omega_k = 0
   \end{cases} \, .
\end{equation}

In order to extend the convergence results in \cite{HLS07, CHL08} for these
iterations, we basically have to prove two facts:
\begin{itemize}
\item[1)] Assumption (14) in \cite{HLS07}.
\item[2)] The {\em local tangential cone condition} \cite[Eq.~(15)]{HLS07},
i.e. the existence of $( {\cal P}_0 , ({\mathbf b_j})_0 ) \in X$ and
$\eta < 1/2$ such that
\begin{multline} \label{eq:a-tcc}
\norm{ F_i( {\cal P} , ({\mathbf b_j}) ) -
   F_i( \bar{\cal P} , (\bar{\mathbf b_j}) ) -
   F_i'( {\cal P} , ({\mathbf b_j}) )
   [ ({\cal P} , ({\mathbf b_j})) - (\bar{\cal P} , (\bar{\mathbf b_j})) ]
   }_{Y}  \leq \\
\eta \norm{ F_i( {\cal P} , ({\mathbf b_j}) ) -
            F_i( \bar{\cal P} , (\bar{\mathbf b_j}) ) }_{Y} \, ,
\end{multline}
for all $( {\cal P} , ({\mathbf b_j}) )$,
$( \bar{\cal P} , (\bar{\mathbf b_j}) ) \in
B_{\rho}( {\cal P}_0 , ({\mathbf b_j})_0 )$, and all $i = 1, \dots, r_{\mathrm{ num}}$.
\end{itemize}
The first one represents no problem. Indeed, the Fr\'echet derivatives of
the operators $F_i$ are locally Lipschitz continuous. Thus, for any
$( {\cal P}_0 , ({\mathbf b_j})_0 ) \in X$ and any $\rho > 0$ we have
$\| F'_i( {\cal P} , ({\mathbf b_j}) ) \| \le M =
M_{\rho,{\cal P}_0,({\mathbf b_j})_0}$ for all
$( {\cal P} , ({\mathbf b_j}) )$ in the ball
$B_{\rho}( {\cal P}_0 , ({\mathbf b_j})_0 ) \subset X$.

The local tangential cone condition however, does not hold. Indeed, the
operators $F_i$ are second order polynomials of the variables $b_{j,n}$
and ${\cal P}$. Therefore, it is enough to verify whether the real
function $f(x,y) = xy$ satisfies 
$$
| f(x,y) - f(\bar{x},\bar{y}) - f'(x,y)((x-\bar{x},y-\bar{y})) |
\leq \eta | f(x,y) - f(\bar{x},\bar{y}) | \, ,
$$
in some vicinity of a point $(x_0,y_0) \in {\mathbb R}^2$ containing a
local minimizer of $f$. This, however, is not the case.

Therefore, the techniques used to prove convergence of the lLK and lSDK
iterations in \cite{HLS07, CHL08} cannot be extended to the nonlinear
system \eqref{eq:inv-probl}.

It is worth noticing that the local tangential cone condition is a standard
assumption in the convergence analysis of adjoint type methods (Landweber,
steepest descent, Levenberg-Marquardt, asymptotical regularization) for
nonlinear inverse problems \cite{EHN96,EL01,DES98,HNS95,KNS08,KS02,Tau94}.
Thus, none of the classical convergence proofs for these iterative methods
can be extended to system \eqref{eq:inv-probl} in a straightforward way.

Motivated by the promising numerical results and efficient performance
of the lLK and lSDK iterations for problems known not to satisfy the
local tangential cone condition (see \cite{HLS07, HKLS07, CHL08}), we
intend to use iteration \eqref{eq:def-iter-nl} for computing approximate
solutions of system \eqref{eq:inv-probl}. This numerical investigation
will be performed in a forthcoming article.

\section{Conclusions} \label{sec:concl}

We presented the description of a discrete mathematical model for Magnetic
Resonance Imaging and derived the corresponding inverse problem for MRI.

We investigate efficient iterative regularization methods for this inverse
problem. An iterative method of Kaczmarz type for obtaining approximate
solutions for the inverse problem is proposed.

Using a particular assumption on the sensitivity kernels, we are able to
prove convergence and stability results for the proposed iterative methods.

The convergence analysis presented in this article extends the results for
the {\em loping Landweber-Kaczmarz} method in \cite{HLS07}. Moreover, we
prove that our method is a convergent iterative regularization method in
the sense of \cite{EHN96}.

\section*{Acknowledgments}
The work of A.L. was supported by the Brazilian National Research Council
CNPq, grants 306020/2006-8 and 474593/2007-0.
J.P.Z. was supported by CNPq under grants 302161/2003-1 and
474085/2003-1.
This work was developed
during the 
permanence of the authors in the Special Semester on Quantitative Biology Analyzed by Mathematical Methods, October 1st, 2007
-January 27th, 2008, organized by RICAM, Austrian Academy of Sciences.


\end{document}